\documentclass[12pt,twoside,reqno]{amsart}
\usepackage[colorlinks=true,citecolor=blue]{hyperref}
\usepackage{mathptmx, amsmath, amssymb, amsfonts, amsthm, mathptmx, enumerate, color}
\usepackage{graphicx}
\usepackage{epstopdf}

\newtheorem{theorem}{Theorem}[section]
\newtheorem{corollary}{Corollary}[section]
\newtheorem{lemma}{Lemma}[section]
\newtheorem{proposition}{Proposition}[section]
\theoremstyle{definition}
\newtheorem{definition}{Definition}[section]
\newtheorem{example}{Example}[section]
\newtheorem{remark}{Remark}[section]

\numberwithin{equation}{section}

\begin{document}
\setcounter{page}{1}

\vspace*{1.0cm}
\title[The behavior of error bounds via Moreau envelopes]
{The behavior of error bounds via Moreau envelopes}
\author[Y. Wang, S.J. Li, Y.H. Hu, M.H. Li, X.B. Li]{ Yu Wang$^{1}$, Shengjie Li$^{2,*}$, Yaohua Hu$^{3}$, Minghua Li$^{4}$, Xiaobing Li$^{5}$}
\maketitle
\vspace*{-0.6cm}

\begin{center}
{\footnotesize {\it

$^1$College of Mathematics and Statistics, Chongqing University, China/Chongqing \\
$^2$College of Mathematics and Statistics, Chongqing University, China/Chongqing \\
$^3$Shenzhen Key Laboratory of Advanced Machine Learning and Applications, College of Mathematics and Statistics, Shenzhen University, China/Shenzhen, Guangdong\\
$^4$The Key Laboratory of Complex Data Analysis and Artificial Intelligence of Chongqing, Chongqing University of Arts and Sciences, China/Yongchuan, Chongqing\\
$^5$College of Mathematics and Statistics, Chongqing Jiaotong University, China/Chongqing \\

}}\end{center}

\vskip 4mm {\small\noindent {\bf Abstract.}
In this paper, we first establish the equivalence of three types of error bounds: uniformized Kurdyka-{\L}ojasiewicz (u-KL) property, uniformized  level-set subdifferential error bound (u-LSEB) and uniformized H\"{o}lder error bound (u-HEB) for prox-regular functions. Then we study the behavior of the level-set subdifferential error bound (LSEB) and the local H\"{o}lder error bound (LHEB) which is expressed respectively by Moreau envelopes, under suitable assumptions. Finally, in order to illustrate our main results and to compare them with those of recent references, some examples are also given.


\renewcommand{\thefootnote}{}
\footnotetext{ $^*$Corresponding author.
\par
E-mail addresses: wangyu103x@163.com (Y. Wang), lisj@cqu.edu.cn (S.J. Li), mayhhu@szu.edu.cn (Y.H. Hu), minghuali20021848@163.com (M.H. Li), xiaobinglicq@126.com (X.B. Li)
\par
%Received January 23, 2015; Accepted February 16, 2016.
}

\section{Introduction}

The theory of error bounds has long been known to be important in variational analysis and optimization theory \cite{{Mordukhovich 2006 2-27},Rockafellar}, and is central to  subdifferential calculus, exact penalty functions, stability and sensitivity analysis, optimality conditions, convergence analysis and convergence rate analysis of various iterative methods \cite{Attouch 2009 6-15-1,Attouch 2013,Burke 2002 1-10,Bolte 2010,Bolte 2014 6-15,Drusvyatskiy 3,Li 5,Li MH 2020,Pang 1997 2-31,Li 6} and references therein. There are many error bounds that have been widely studied in recent years, such as the Kurdyka-{\L}ojasiewicz (KL) inequality, see \cite{Bolte 2010,Li 5,Li 6,Bolte 2017 1-7,Bai 1,Li 2016}, which originated from the seminal work \cite{Lojasiewicz 1963} on {\L}ojasiewicz ({\L}) inequality; the level-set subdifferential error bound (LSEB), see \cite{Bai 1,zhu 2}, which is weaker than that of the KL inequality and the notion of LSEB was first introduced in \cite{zhu 2} and can replace the KL inequality to establish linear convergence for a number of numerical optimization algorithms for non-convex and non-smooth optimization problems;  and the local H\"{o}lder error bound (LHEB), see \cite{Drusvyatskiy 3,Bai 1,Kruger 2019 2-23,Burke 1993}, which includes the traditional linear error bound and weaker than LSEB, and in some subtle applications the traditional linear error bound does not hold while the LHEB does, see \cite{Ioffe 2017,Ngai 2017,Yao 2016}.

Regarding the relationship between these three error bounds, Bai et al. \cite{Bai 1} established the equivalence among {\L} inequality, LSEB and weak sharp minima for weakly convex functions with exponent $\alpha\in[0,1]$ and approximately convex functions. This generalises the equivalence between the LSEB and the {\L} inequality under the convexity assumption in \cite{zhu 2}. In summary, the above error bounds are given for a stationary point and not for the stationary set.

With respect to the KL exponent, Li and Pong \cite{Li 5} studied the KL exponent, which is an important quantity for analysing the convergence rate of first-order methods, and they developed various calculus rules to deduce the KL exponent of new (possibly non-convex and non-smooth) functions formed from functions with known KL exponents. In \cite{Li 6}, Yu et al. showed that under mild assumptions the KL exponent is preserved via inf-projection and that the KL exponent of the Bregman envelope of a function is restrained by that of the function itself, improving on the corresponding results in \cite{Li 5}.

Moreau envelope plays an important role in the structure of the algorithm and was first mentioned in \cite{Moreau 1962} and has since been studied by many scholars, cf. \cite{Drusvyatskiy 3,Li 5,Li 6,Rockafellar 2009,Jourani 2014,Kan 2012,Perez-Aros 2021,Rockafellar 1976} and references therein. Noting that the Moreau envelope is a special case of the Bregman envelope, and that there are few studies on whether the LSEB and LHEB exponents have transferability to the Moreau envelope.

Motivated by \cite{Li 5,Li 6,Bai 1,zhu 2}, in this paper, we first establish the equivalence among u-KL property, u-LSEB and u-HEB on a nonempty and compact set for a prox-regular function, which is a generalisation of the corresponding conclusion in \cite{Bai 1}. Then we show that LSEB (u-LSEB) and LHEB (u-HEB) exponents and constants are expressed via the Moreau envelopes of a function is restrained by that of the function itself, which improves the corresponding results in \cite{Li 6}. We also give some examples to show that our results are more applicable.

The rest of this paper is structured as follows. Section \ref{2} provides notation and preliminaries. Section \ref{3.1} investigates the equivalence among the u-KL property, u-LSEB and u-HEB. Section \ref{4} presents the different behavior under Moreau envelopes of a proper and closed function. Finally, we present some concluding remarks in Section \ref{5}.

\section{Preliminaries}\label{2}

Throughout this paper, we use $\mathbb{R }^n$ to denote the n-dimensional Euclidean space with scalar product $\langle\cdot ,\cdot \rangle$ and corresponding norm $\|\cdot\|$.
Let $B(c, r)$ denote the open ball with centered at $c$ and radius $r$. For a nonempty set $A\subseteq \mathbb{R }^n$, we denote the distance from a point $x\in \mathbb{R }^n$ to $A$ by
$$d(x,A)=\inf_{y\in A} \|x -y\|,$$
the set of points in $A$ that achieve this infimum is called the projection of $x$ onto $A$ and denoted by $\textrm{Proj}_A(x)$. For a empty set $A$, we define $d(x, A) = +\infty$.

An extended real valued function $f : \mathbb{R }^n \rightarrow \bar{\mathbb{R}}:= (-\infty,\infty]$ is called proper if its domain $\textrm{dom}\,f := \big\{x : f(x) < \infty\big\}\neq \emptyset$. The function $f$ is called closed if it is lower semicontinuous $\big(\liminf_{x\rightarrow \bar{x}}f(x)=f(\bar{x})\,\,\textrm{for any}\,\, \bar{x}\in \mathbb{R }^n \big )$.%
For a closed function $f : \mathbb{R }^n \rightarrow \bar{\mathbb{R}}$ and $\alpha\in \mathbb{R}$, the level set of $f$ is the closed set
$$\textrm{lev}_{\leq \alpha}f:=\big\{x\in \mathbb{R }^n: f(x)\leq \alpha \big\}.$$
A proper and closed function $f : \mathbb{R }^n \rightarrow \bar{\mathbb{R}}$ is said to satisfy a Lipschitz condition of rank $K$ on a given set $S$ provided that $f$ is finite on $S$ and satisfies
$$|f(x)-f(y)|\leq \|x-y\|,\,\,\forall x,y \in S.$$
A proper and closed function $f:\mathbb{R}^n \rightarrow \bar{\mathbb{R}}$ is said to be prox-regular at $\bar{x}$ for $\bar{v}$ with a constant $\rho$, if $\bar{x}\in \textrm{dom} \,f$ with $\bar{v}\in \partial f(\bar{x})$, and there exist $\varepsilon>0$ and $\rho\geq 0$ such that
\begin{eqnarray}
   f(y)\geq f(x) +\langle v,y-x\rangle-\frac{\rho}{2}\|y-x\|^2  \,\, \forall y\in B(\bar{x},\varepsilon)\nonumber
\end{eqnarray}
when $v\in \partial f(x)\cap B(\bar{v},\varepsilon)$, $x\in B(\bar{x},\varepsilon)$ with $f(x)<f(\bar{x})+\varepsilon$.

The regular subdifferential and limiting subdifferential of $f$ at $\bar{x}\in \textrm{dom}\,f$ are defined in [46, Definition 8.3], respectively, by
$$\hat{\partial} f(\bar{x}):=\bigg\{v\in \mathbb{R }^n:\liminf_{x\rightarrow \bar{x},x\neq \bar{x}}\frac{f(x)-f(\bar{x})-\langle v,x-\bar{x}\rangle}{\|x-\bar{x}\|}\geq 0\bigg\},$$
$$\partial f(\bar{x}):= \big\{v\in \mathbb{R }^n: \exists x_k \stackrel{f}{\rightarrow} \bar{x}, v_k\rightarrow v \,\,\textrm{with}\,\, v_k\in \hat{\partial} f(x_k) \,\,\textrm{for each}\,\, k\big\},$$
where $x_k \stackrel{f}{\rightarrow} \bar{x}$ means both $x_k \rightarrow \bar{x}$ and $f(x_k) \rightarrow f(\bar{x})$. By the definitions of subdifferentials, the inclusion relationship $\hat{\partial} f(x)\subseteq \partial f(x)$ always holds. By convention, we set $\partial f(x)=\emptyset$ for $x \notin \textrm{dom} \,f$, and write $\textrm{dom} \,\partial f:= \{x \in\mathbb{R }^n: \partial f(x)\neq\emptyset\}$.

For simplicity, for any $\bar{x}\in\textrm{dom}\,f$ and $\nu,\eta>0$,% take
$$[f(\bar{x})<f<f(\bar{x})+\nu]:=\big\{x\in\mathbb{R }^n:f(\bar{x})<f(x)<f(\bar{x})+\nu  \big\},$$
$$\mathfrak{B}(\bar{x},\eta,\nu):= B(\bar{x},\eta)\cap [f(\bar{x})<f<f(\bar{x})+\nu].$$

A proper and closed function $f: \mathbb{R }^n \rightarrow \bar{\mathbb{R}}$ satisfies the Kurdyka-{\L}ojasiewicz (KL) property or {\L}ojasiewicz inequality at $\bar{x}\in\textrm{dom} \,\partial f$ with an exponent $\gamma\in [0,1)$ and a constant $\mu> 0$, if there exist $\eta,\nu > 0$ such that
\begin{eqnarray}
\big(f(x)-f(\bar{x})\big)^\gamma\leq \mu d\big(0,\partial f(x)\big)\,\, \forall x\in\mathfrak{B}(\bar{x},\eta,\nu).\nonumber
\end{eqnarray}

The function $f$ satisfies the level-set subdifferential error bound (LSEB) at $\bar{x}\in\textrm{dom} \,\partial f$ with an exponent $\gamma\geq 0$ and a constant $\mu> 0$, if there exist $\eta > 0$ and $\nu > 0$ such that
\begin{eqnarray}
d^{\gamma}(x,\textrm{lev}_{\leq f(\bar{x})}f)\leq \mu d\big(0,\partial f(x)\big)\,\, \forall x\in\mathfrak{B}(\bar{x},\eta,\nu).\nonumber
\end{eqnarray}

And the function $f$ satisfies the local H\"{o}lder error bound (LHEB) at $\bar{x}\in\textrm{dom} \, f$ with an exponent $\gamma> 0$ and a constant $\mu> 0$, % in the Hoffman's sense,
    if there exist $\eta > 0$ such that
$$d^{\gamma}(x,\textrm{lev}_{\leq f(\bar{x})}f)\leq \mu \big(f(x)-f(\bar{x})\big)$$
for any $x\in B(\bar{x},\eta)$ with $f(\bar{x})<f(x)$.

Now, we recall the uniformized KL property, which can be found in \cite[Lemma 2.2]{Li 6}.
\begin{definition}\label{def2.1}\rm
 Suppose that $\Omega\subseteq\textrm{dom}\,\partial f$ is a nonempty and compact set and $f$ takes a constant value on $\Omega$. The function $f$ satisfies the uniformized Kurdyka-{\L}ojasiewicz (u-KL) property on $\Omega$ with an exponent $0\leq\gamma<1$ and a constant $\mu> 0$, if there exist $\varepsilon > 0$ and $\nu > 0$ such that
\begin{eqnarray}
\big(f(x)-f(\bar{x})\big)^\gamma\leq \mu d\big(0,\partial f(x)\big)\nonumber
\end{eqnarray}
for any $\bar{x}\in\Omega$ and any $x$ with $d(x,\Omega)<\varepsilon$ and $f(\bar{x})<f(x)<f(\bar{x})+\nu$.
\end{definition}

Motivated by Definition \ref{def2.1}, we introduce the following definitions.
\begin{definition}\rm
Suppose that $\Omega\subseteq\textrm{dom}\,\partial f$ is a nonempty and compact set and $f$ takes a constant value on $\Omega$. The function $f$ satisfies the uniformized level-set subdifferential error bound (u-LSEB) on $\Omega$ with an exponent $\gamma\geq 0$ and a constant $\mu> 0$, if there exist $\varepsilon > 0$ and $\nu > 0$ such that
\begin{eqnarray}
d^{\gamma}(x,\textrm{lev}_{\leq f(\bar{x})}f)\leq \mu d\big(0,\partial f(x)\big)\nonumber
\end{eqnarray}
for any $\bar{x}\in\Omega$ and any $x$ with $d(x,\Omega)<\varepsilon$ and $f(\bar{x})<f(x)<f(\bar{x})+\nu$.
\end{definition}

\begin{definition}\rm
Suppose that $\Omega\subseteq\textrm{dom}\, f$ is a nonempty and compact set and $f$ takes a constant value on $\Omega$. The function $f$ satisfies the uniformized H$\ddot{\textrm{o}}$lder error bound (u-HEB) on $\Omega$ with an exponent $\gamma > 0$ and a constant $\mu> 0$, if there exist $\varepsilon > 0$ such that
\begin{eqnarray}
d^{\gamma}(x,\textrm{lev}_{\leq f(\bar{x})}f)\leq \mu \big(f(x)-f(\bar{x})\big)\nonumber
\end{eqnarray}
for any $\bar{x}\in\Omega$ and any $x$ with $d(x,\Omega)<\varepsilon$ and $f(\bar{x})<f(x)$.
\end{definition}

\section{The equivalence among u-KL property, u-LSEB and u-HEB}\label{3.1}

In this section we want to discuss the equivalence among u-KL property, u-LSEB and u-HEB. Before doing so, we recall Lemma 2.2 in \cite{Li 6}, which concerned the u-KL property. Then the following lemmas are given.

\begin{lemma} \label{Lemma 1}
Let $f:\mathbb{R}^n \rightarrow \bar{\mathbb{R}}$ be a proper and closed function and $\Omega \subseteq \textrm{dom}\,\partial f$ be a nonempty and compact set. If $f$ takes a
constant value on $\Omega$ and satisfies the LSEB at each point of $\Omega$. Then $f$ satisfies the u-LSEB on $\Omega$ with an exponent $\gamma\geq 0$ and a constant $\mu> 0$.
\end{lemma}

\begin{proof} For any $\bar{x}\in\Omega$. Since $f$ takes a constant value on $\Omega$ and satisfies the LSEB at each point of $\Omega$, then for any $z\in \Omega$, there exist $\gamma_z\geq 0$ and $\mu_z,\varepsilon_z,\nu_z> 0$ such that
\begin{eqnarray}
d^{\gamma_z}(x,\textrm{lev}_{\leq f(\bar{x})}f)\leq \mu_z d\big(0,\partial f(x)\big)\,\,\forall x\in \mathfrak{B}(z,\varepsilon_z,\nu_z).\nonumber
\end{eqnarray}

Hence, $\{B(z_i,\frac{\varepsilon_{z_i}}{2}):z_i\in \Omega\}$ is an open cover of $\Omega$, it follows from the compactness of $\Omega$ that there exist $z_i\in\Omega$, $i = 1,...,p$ such that
\begin{eqnarray}
\Omega\subseteq \bigcup^p_{i=1}B(z_i,\frac{\varepsilon_{z_i}}{2}).\label{2.2}
\end{eqnarray}
Set $\nu := \min \{\nu_{z_i}: i = 1,..., p\}$, $\varepsilon := \min \{\varepsilon_{z_i}: i = 1,..., p\}$ and $U_\varepsilon:=\{x\in \mathbb{R}^n:d\big(x,\Omega)<\frac{\varepsilon}{2}\}.$ % we choose a point $x_i\in \textrm{Proj}_\Omega(x)$,
For any $x\in U_\varepsilon$, it follows from (\ref{2.2}) there is an index $j\in\{1,...p\}$ such that $x_j\in B(z_j,\frac{\varepsilon_{z_j}}{2})$, where $x_j$ is one of the projection of $x$ onto $\Omega$. Then we have
%For any $x\in U_\varepsilon$, it follows from (\ref{2.2}) there is an index $i\in\{1,...p\}$ such that one of the point $x_i$ of the projection of $x$ onto $\Omega$ with $x_i\in B(z_i,\frac{\varepsilon_{z_i}}{2})$. Then we have
\begin{equation}
\|x-z_j\|\leq \|x-x_j\|+\|x_j-z_j\| < \frac{\varepsilon}{2}+\frac{\varepsilon_{z_j}}{2}\leq \varepsilon_{z_j}. \nonumber
\end{equation}
It means that $U_\varepsilon\subseteq \bigcup^p_{i=1}B(z_i,\varepsilon_{z_i})$ and
$$U_\varepsilon\cap[f(\bar{x})<f<f(\bar{x})+\nu]\subseteq \bigcup^p_{i=1}\bigg\{ B(z_i,\varepsilon_{z_i})\cap [f(\bar{x})<f<f(\bar{x})+\nu_{z_i}]\bigg\}=\bigcup^p_{i=1}\mathfrak{B}(z_i,\varepsilon_{z_i},\nu_{z_i}).$$
Then, for any $x \in U_\varepsilon\cap[f(\bar{x})<f<f(\bar{x})+\nu]$, one has
\begin{equation}
d^{\gamma}(x,\textrm{lev}_{\leq f(\bar{x})}f)\leq \mu d\big(0,\partial f(x)\big),\nonumber
\end{equation}
where $\gamma := \max \{\gamma_{z_i}: i = 1,..., p\}$ and $\mu := \max \{\mu_{z_i}: i = 1,..., p\}$. The proof is complete. \end{proof}

\begin{lemma} \label{Lemma 2}
Let $f:\mathbb{R}^n \rightarrow \bar{\mathbb{R}}$ be a proper and closed function and $\Omega \subseteq \textrm{dom}\,f$ be a nonempty and compact set. If $f$ takes a
constant value on $\Omega$ and satisfies the LHEB at each point of $\Omega$. Then $f$ satisfies the u-HEB on $\Omega$ with an exponent $\gamma> 0$ and a constant $\mu> 0$.
\end{lemma}

\begin{proof} For any $\bar{x}\in\Omega$. The compact set $\Omega$ can be covered by a finite number of open balls $B(x_i,\varepsilon_i)$ with $x_i\in\Omega$ and $\varepsilon_i>0$ for $i = 1,...,p$. Since $f$ satisfies the LHEB at each point of $\Omega$, then there exist $\gamma_i> 0$ and $\mu_i  > 0$ such that
\begin{eqnarray}
d^{\gamma_i}(x,\textrm{lev}_{\leq f(\bar{x})}f)\leq \mu_i \big(f(x)-f(\bar{x})\big) \nonumber
\end{eqnarray}
for any $x\in B(x_i,\varepsilon_i)$ with $f(\bar{x})<f(x)$.
We can choose $\varepsilon>0$ be small enough so that
$$U_\varepsilon:=\{x\in \mathbb{R}^n:d\big(x,\Omega)\leq\varepsilon\}\subseteq \bigcup^p_{i=1}B(x_i,\varepsilon_i).$$
Then, for every $x \in U_\varepsilon$ with $f(\bar{x})<f(x)$, we have
$$d^{\gamma}(x,\textrm{lev}_{\leq f(\bar{x})}f)\leq \mu \big(f(x)-f(\bar{x})\big),$$
where $\gamma := \max \{\gamma_i: i = 1,..., p\}$ and $\mu := \max \{\mu_i: i = 1,..., p\}$. The proof is completed.  \end{proof}

Now, we recall the following lemma which follows from the Theorem 2.1 in \cite{Bai 1}.

\begin{lemma} \label{{Bai 1} Theorem 2.1 (b)} \cite{Bai 1}
Let $f$ be a proper and closed function and $\bar{x}\in \mathbb{R}^n$. Assume that $f$ satisfies the KL property at $\bar{x}$ with an exponent $\gamma_1 \in [0, 1)$ and a constant $\mu_1> 0$. Then, $f$ satisfies LSEB at $\bar{x}$ with an exponent $\gamma_2=\frac{\gamma_1}{1-\gamma_1}$ and a constant $\mu_2=(1-\gamma_1)^{\frac{-\gamma_1}{1-\gamma_1}} {\mu_1}^{\frac{1}{1-\gamma_1}}$. Furthermore, $f$ satisfies LHEB at $\bar{x}$ with an exponent $\gamma_3=\gamma_2+1$ and a constant $\mu_3=\frac{(\gamma_2+1)^{\gamma_2+1}}{\gamma_2^{\gamma_2}}\mu_2$.
\end{lemma}

Motivated by Theorem 2.1 of \cite{Bai 1}, we give the following result.
\begin{proposition} \label{prop 2.1}
Let $f:\mathbb{R}^n\rightarrow \bar{\mathbb{R}}$ be a proper and closed function and $\bar{x}\in \partial^{-1} f(0)$. Assume that $f$ is prox-regular at $\bar{x}$ for $\bar{v}=0$ with a constant $\rho\geq 0$. Consider the following conditions:
\begin{enumerate}
  \item [(i)] $f$ satisfies the KL property at $\bar{x}$ with an exponent $1>\gamma_1\geq 0$ and a constant $\mu_1> 0$;
  \item [(ii)] $f$ satisfies the LSEB at $\bar{x}$ with an exponent $\gamma_2\geq 0$ and a constant $\mu_2> 0$;
  \item [(iii)] $f$ satisfies the LHEB at $\bar{x}$ with an exponent $\gamma_3> 0$ and a constant $\mu_3> 0$.
\end{enumerate}
Then, we have the following results:
\begin{enumerate}
  \item [(a)] If $1\leq \gamma_3 <2 $, (iii) $\Rightarrow$ (ii) with the exponent $\gamma_2=\gamma_3-1$ and the constant $\mu_2>\mu_3$; If $\gamma_3= 2$ and $\mu_3 <\frac{2}{\rho}$, (iii) $\Rightarrow$ (ii) with the exponent $\gamma_2=1$ and the constant $\mu_2=\frac{2\mu_3}{2-\rho\mu_3}$.
  \item [(b)] (ii) $\Rightarrow$ (i) with the exponent $\gamma_1=\max\{\frac{\gamma_2}{1+\gamma_2},\frac{\gamma_2}{2}\}$ and the constant $\mu_1=(\mu_2^{\frac{1}{\gamma_2}}+\frac{\rho}{2}\mu_2^{\frac{2}{\gamma_2}})^{\max\{\frac{\gamma_2}{1+\gamma_2},\frac{\gamma_2}{2}\}}$.

\end{enumerate}
\end{proposition}

\begin{proof} (a) The proof is similar to Theorem 2.1 in \cite{Bai 1} and we omit it.
(b) Assume that $f$ does not have the KL property at $\bar{x}$ with the exponent $\gamma_1=\max\{\frac{\gamma_2}{1+\gamma_2},\frac{\gamma_2}{2}\}$ and the constant $\mu_1= (\mu_2^{\frac{1}{\gamma_2}}+\frac{\rho}{2}\mu_2^{\frac{2}{\gamma_2}})^{\max\{\frac{\gamma_2}{1+\gamma_2},\frac{\gamma_2}{2}\}}$. Then there exists a sequence $\{x_k\}$ with $x_k\rightarrow \bar{x}$ and $f(x_k)\downarrow f(\bar{x})$ such that
\begin{eqnarray}
\big(f(x_k)-f(\bar{x})\big)^{\gamma_1}> \mu_1 d\big(0,\partial f(x_k)\big).\label{t5}
\end{eqnarray}
Since $\partial f(x_k)$ is closed, we can choose that $v_k\in \partial f(x_k)$ with $d\big(0,\partial f(x_k)\big)=\|v_k\| \rightarrow 0$. Then by (\ref{t5}) one has
\begin{eqnarray}
\|v_k\|^{\max\{\frac{\gamma_2}{1+\gamma_2},\frac{\gamma_2}{2}\}^{-1}}<\mu_1^{-\max\{\frac{\gamma_2}{1+\gamma_2},\frac{\gamma_2}{2}\}^{-1}} \big(f(x_k)-f(\bar{x})\big).   \label{t5.5}
\end{eqnarray}
Since $f$ is prox-regular at $\bar{x}$ for $\bar{v}=0$ with a constant $\rho\geq 0$, there exist $\varepsilon>0$ such that
\begin{eqnarray}
   f(y)\geq f(x) +\langle v,y-x\rangle-\frac{\rho}{2}\|y-x\|^2  \,\, \forall y\in B(\bar{x},\varepsilon)\nonumber%\label{t6.5}
\end{eqnarray}
when $v\in \partial f(x)\cap B(0,\varepsilon)$, $x\in B(\bar{x},\varepsilon)$ with $f(x)<f(\bar{x})+\varepsilon$. Let $y_k\in \textrm{Proj}_{\textrm{lev}_{\leq f(\bar{x})}f}(x_k)$. Then we have $y_k\rightarrow \bar{x}$ and $\|y_k-x_k\|\rightarrow 0$ as $k\rightarrow \infty$ due to
$\|y_k-\bar{x}\|\leq \|y_k-x_k\|+\|x_k-\bar{x}\|\leq 2\|x_k-\bar{x}\|$ and $x_k\rightarrow \bar{x}$.
Then there exists $k$ large enough such that $x_k,y_k\in B(\bar{x},\varepsilon)$, $f(x_k)<f(\bar{x})+\varepsilon$, $v_k\in \partial f(x_k)$ with $d\big(0,\partial f(x_k)\big)=\|v_k\|<\varepsilon$ and
\begin{eqnarray}
   f(y_k)\geq f(x_k) +\langle v_k,y_k-x_k\rangle-\frac{\rho}{2}\|y_k-x_k\|^2.\nonumber
\end{eqnarray}
Then we have
\begin{equation}
\aligned d\big(0,\partial f(x_k)\big)&\geq\frac{f(x_k)-f(y_k)}{\|y_k-x_k\|}-\frac{\rho}{2}\|y_k-x_k\|\nonumber\\
&\geq\frac{f(x_k)-f(\bar{x})}{\|y_k-x_k\|}-\frac{\rho}{2}\|y_k-x_k\|.\label{t7}
\endaligned
\end{equation}
Since $f$ satisfies the LSEB at $\bar{x}$ with an exponent $\gamma_2\geq 0$ and a constant $\mu_2> 0$, we can choose suitable neighborhood of $\bar{x}$ such that
\begin{eqnarray}
d^{\gamma_2}(x_k,\textrm{lev}_{\leq f(\bar{x})}f)=\|y_k-x_k\|^{\gamma_2}\leq \mu_2 d\big(0,\partial f(x_k)\big)=\mu_2\|v_k\|,\nonumber
\end{eqnarray}
which means that $\|y_k-x_k\|\leq \mu_2^{\frac{1}{\gamma_2}} \|v_k\|^{\frac{1}{\gamma_2}}$. Combining (\ref{t5.5}), (\ref{t7}), $\gamma_1=\max\{\frac{\gamma_2}{1+\gamma_2},\frac{\gamma_2}{2}\}$ and $\mu_1= (\mu_2^{\frac{1}{\gamma_2}}+\frac{\rho}{2}\mu_2^{\frac{2}{\gamma_2}})^{\max\{\frac{\gamma_2}{1+\gamma_2},\frac{\gamma_2}{2}\}}$, one has
\begin{equation}
\aligned  f(x_k)-f(\bar{x})&\leq  \mu_2^{\frac{1}{\gamma_2}}\|v_k\|^{\frac{1+\gamma_2}{\gamma_2}}+\frac{\rho}{2}\mu_2^{\frac{2}{\gamma_2}}\|v_k\|^{\frac{2}{\gamma_2}}\nonumber\\
 &\leq\big(\mu_2^{\frac{1}{\gamma_2}}+\frac{\rho}{2}\mu_2^{\frac{2}{\gamma_2}}\big)\|v_k\|^{\min\{\frac{1+\gamma_2}{\gamma_2},\frac{2}{\gamma_2}\}}\nonumber\\
 &=\big(\mu_2^{\frac{1}{\gamma_2}}+\frac{\rho}{2}\mu_2^{\frac{2}{\gamma_2}}\big)\|v_k\|^{\max\{\frac{\gamma_2}{1+\gamma_2},\frac{\gamma_2}{2}\}^{-1}}\nonumber\\
 &<\big(\mu_2^{\frac{1}{\gamma_2}}+\frac{\rho}{2}\mu_2^{\frac{2}{\gamma_2}}\big)\mu_1^{-\max\{\frac{\gamma_2}{1+\gamma_2},\frac{\gamma_2}{2}\}^{-1}} \big(f(x_k)-f(\bar{x})\big)\nonumber\\
 &=f(x_k)-f(\bar{x}),\nonumber
\endaligned
\end{equation}
which is a contradiction. The proof is completed.  \end{proof}

\begin{remark}
Since the weakly convex functions are prox-regular functions, Proposition \ref{prop 2.1} generalizes Theorem 2.1 in \cite{Bai 1}. And we improve the proof of Theorem 2.1 (e) and (f) in \cite{Bai 1} so that the LSEB exponent is not bounded by the weakly convex exponent.
\end{remark}

Now we consider the equivalence among u-KL property, u-LSEB and u-HEB under suitable assumptions.

\begin{theorem} Let $f:\mathbb{R}^n \rightarrow \bar{\mathbb{R}}$ be a proper and closed function and $\Omega \subseteq \textrm{dom}\,\partial f$ be a nonempty and compact set. Suppose that $f$ takes a constant value on $\Omega$. Consider the following conditions:
\begin{enumerate}
  \item [(i)] $f$ satisfies the u-KL property on $\Omega$ with an exponent $0\leq \gamma_1<1$ and a constant $\mu_1> 0$;
  \item [(ii)] $f$ satisfies the u-LSEB on $\Omega$ with an exponent $\gamma_2\geq 0$ and a constant $\mu_2> 0$;
  \item [(iii)] $f$ satisfies the u-HEB on $\Omega$ with an exponent $\gamma_3> 0$ and a constant $\mu_3> 0$.
\end{enumerate}
Then, we have the following results:
\begin{enumerate}
  \item [(a)] (i) $\Rightarrow$ (ii) with the exponent $\gamma_2=\frac{\gamma_1}{1-\gamma_1}$ and the constant $\mu_2=(1-\gamma_1)^{\frac{-\gamma_1}{1-\gamma_1}}\mu_1^{\frac{1}{1-\gamma_1}}$;
  \item [(b)] (ii) $\Rightarrow$ (iii) with the exponent $\gamma_3=\gamma_2+1$ and the constant $\mu_3=\frac{(\gamma_2+1)^{\gamma_2+1}}{\gamma_2^{\gamma_2}}\mu_2$.
\end{enumerate}
Moreover, if $f$ is prox-regular at every point of $\Omega$ with a constant $\rho\geq 0$, then we have the following results:
\begin{enumerate}
  \item [(c)] If $1\leq \gamma_3 <2 $, (iii) $\Rightarrow$ (ii) with the exponent $\gamma_2=\gamma_3-1$ and the constant $\mu_2>\mu_3$; If $\gamma_3= 2$ and $\frac{\rho}{2}<\frac{1}{\mu_3}$, (iii) $\Rightarrow$ (ii) with the exponent $\gamma_2=1$ and the constant $\mu_2>\frac{2\mu_3}{2-\rho\mu_3}$.
  \item [(d)](ii) $\Rightarrow$ (i) with the exponent $\gamma_1=\max\{\frac{\gamma_2}{1+\gamma_2},\frac{\gamma_2}{2}\}$ and the constant $\mu_1=(\mu_2^{\frac{1}{\gamma_2}}+\frac{\rho}{2}\mu_2^{\frac{2}{\gamma_2}})^{\max\{\frac{\gamma_2}{1+\gamma_2},\frac{\gamma_2}{2}\}}$.
\end{enumerate}
\end{theorem}

\begin{proof} By Lemmas \ref{Lemma 1} and \ref{{Bai 1} Theorem 2.1 (b)}, (a) is established immediately. (b) holds from the Lemmas \ref{Lemma 2} and \ref{{Bai 1} Theorem 2.1 (b)}. (c) holds from Proposition \ref{prop 2.1} (a) and Lemma \ref{Lemma 1}, and (d) follows from Proposition \ref{prop 2.1} (b) and Lemma \ref{Lemma 2}.\end{proof}

\section{Behavior of error bounds via Moreau envelopes}\label{4}

In this section, we consider the Moreau envelopes and proximal mappings which defined in \cite[Definition 1.22]{Rockafellar}. For a proper and closed function $f : \mathbb{R }^n \rightarrow \bar{\mathbb{R}}$ and parameter $\lambda> 0$, the Moreau envelope $e_{\lambda }f$ and proximal mapping $P_{\lambda}f(x)$ are defined, respectively, by
$$e_{\lambda }f(x):=\inf_{y\in \mathbb{R}^n}\big\{f(y)+\frac{1}{2\lambda}\parallel y-x \parallel^2\big\}\leq f(x),$$
\begin{eqnarray}
P_{\lambda}f(x):=\arg \min _{y\in \mathbb{R}^n}\big\{f(y)+\frac{1}{2\lambda}\parallel y-x \parallel^2\big\}.\nonumber
\end{eqnarray}

A function $f : \mathbb{R }^n \rightarrow \bar{\mathbb{R}}$ is prox-bounded if there exists $\lambda > 0$ such that $e_{\lambda }f(x) > -\infty$ for some $x \in \mathbb{R }^n  $. The supremum
of the set of all such $\lambda$ is the threshold $\lambda_f$ of prox-boundedness for $f$, see \cite[Definition 1.23]{Rockafellar}.

It immediately follows from Lemma 2.1 in \cite{Li 5} and Lemma \ref{{Bai 1} Theorem 2.1 (b)} that $e_{\lambda}f$ satisfies the KL property with an exponent 0, the LSEB with an exponent 0 and the LHEB with an exponent 1 at $x$ with $0\notin \partial e_{\lambda}f(x)$. So we only consider the point $\bar{x}$ with $0\in \partial e_{\lambda}f(\bar{x})$ in what follows.

A local minimum occurs at $\bar{x}\in \textrm{dom}\,f$ if $f(x) \geq f(\bar{x})$ for all $x \in V$, where $V$ is a neighborhood of $\bar{x}$. At first, we consider the behavior of LSEB via Moreau envelopes.

\begin{theorem} \label{th4.1}
Let $f:\mathbb{R}^n \rightarrow \bar{\mathbb{R}}$ be a proper and closed function and prox-bounded with threshold $\lambda_f$. Assume that the following conditions hold:
\begin{enumerate}
  \item [(i)] For $\lambda\in (0, \lambda_f )$, $\bar{x}\in \textrm{dom}\,f$ with $0\in \partial e_{\lambda}f(\bar{x})$;
  \item [(ii)] $f$ satisfies the LSEB at $\bar{x}$ with an exponent $\gamma> 0$ and a constant $\mu>0$.%\in (0,\lambda]$.
\end{enumerate}
Then $e_{\lambda }f$ satisfies the LSEB at $\bar{x}$ with an exponent $\max\{1,\gamma\}$ and a constant $\lambda\big(1+(\frac{\mu}{\lambda})^{\gamma^{-1}}\big)^{\max\{1,\gamma\}}$.
\end{theorem}

\begin{proof} For any $\bar{x}\in \textrm{dom}\,f$ with $0\in \partial e_{\lambda}f(\bar{x})$, we have $0 \in \lambda^{-1}\big(\bar{x}-P_{\lambda}f(\bar{x})\big)$ by the condition (i) and Example 10.32 in \cite{Rockafellar}. Then we get $\bar{x}\in P_{\lambda}f(\bar{x})$ and $f(\bar{x})= e_{\lambda }f(\bar{x})$.

Next, since $f$ satisfies the LSEB at $\bar{x}$ with an exponent $\gamma> 0$ and a constant $\mu>0$, there exist $\eta,\nu >0$ such that
\begin{equation}
d^{\gamma }(x,\textrm{lev}_{\leq f(\bar{x})}f)\leq \mu d\big(0,\partial f(x)\big)\,\,\forall x\in \mathfrak{B}(\bar{x},\eta,\nu). \label{4.1}
\end{equation}

For any $x\in B(\bar{x},\tilde{\eta})$ with $e_{\lambda }f(\bar{x})<e_{\lambda }f(x)<e_{\lambda }f(\bar{x})+\tilde{\nu}$, where $\tilde{\eta}\in (0,\frac{\eta}{2})$ and $\tilde{\nu}\in (0,\nu)$, there exists a point $y\in P_{\lambda}f(x)$ such that $d\big(x,P_{\lambda}f(x)\big)=\parallel x-y \parallel$. Obviously, one has
$$f(y)\leq e_{\lambda }f(x)<e_{\lambda }f(\bar{x})+\tilde{\nu}<f(\bar{x})+\nu$$
 and $\frac{1}{\lambda}(x-y)\in \partial f(y)$.
If $f(y)\leq f(\bar{x})$, it means that $y\in\textrm{lev}_{\leq f(\bar{x})}f$ and
\begin{equation}
d(x,\textrm{lev}_{\leq f(\bar{x})}f)\leq  \parallel x-y \parallel .\label{4.2}
\end{equation}
If $f(\bar{x})< f(y)<f(\bar{x})+\nu$, note that $y\in P_{\lambda}f(x)$, we have
$$f(y)+\frac{1}{2\lambda}\parallel y-x \parallel^2\leq f(\bar{x})+\frac{1}{2\lambda}\parallel \bar{x}-x \parallel^2.$$
It follows from $f(\bar{x})<f(y)$ and $\lambda>0$ that we have $\parallel y-x \parallel< \parallel x-\bar{x} \parallel<\frac{\eta}{2}$. Hence we have $\parallel y-\bar{x} \parallel\leq \parallel y-x  \parallel+\parallel x-\bar{x} \parallel<\eta,$ which means that $y\in \mathfrak{B}(\bar{x},\eta,\nu)$ and (\ref{4.1}) holds with $y$. Then for any such $x$, we have
\begin{align}
d(x,\textrm{lev}_{\leq f(\bar{x})}f)\leq& \parallel x-y\parallel+d(y,\textrm{lev}_{\leq f(\bar{x})}f)\nonumber\\
\leq&\parallel x-y\parallel+ \big(\mu d\big(0,\partial f(y)\big)\big)^{\gamma^{-1}}\nonumber\\
\leq& \parallel x-y\parallel+\big(\frac{\mu}{\lambda}\parallel x-y\parallel \big)^{\gamma^{-1}}. \label{4.3}
\end{align}
The last inequality holds with $\frac{1}{\lambda}(x-y)\in \partial f(y)$. Combining (\ref{4.2}) and (\ref{4.3}), for any $x\in B(\bar{x},\tilde{\eta})$ with $e_{\lambda }f(\bar{x})<e_{\lambda }f(x)<e_{\lambda }f(\bar{x})+\tilde{\nu}$ and $y\in \textrm{Proj}_{P_{\lambda}f(x)}(x)$, one has
\begin{align}
d(x,\textrm{lev}_{\leq f(\bar{x})}f)\leq \parallel x-y\parallel+\big(\frac{\mu}{\lambda}\parallel x-y\parallel \big)^{\gamma^{-1}}.\nonumber
\end{align}

Shrink $\eta$ if necessary so that $\parallel x-y \parallel< \eta\leq 1$. If $0< \gamma <1$, we have
\begin{align}
d(x,\textrm{lev}_{\leq f(\bar{x})}f)\leq \big(1+(\frac{\mu}{\lambda})^{\gamma^{-1}}\big)\parallel x-y\parallel.\nonumber
\end{align}
Set $\tilde{\mu}:=\lambda\big(1+(\frac{\mu}{\lambda})^{\gamma^{-1}}\big)$. %it follows from $\mu\in (0,\lambda]$ that we have $\tilde{\mu}< +\infty$.
Then for any $x\in B(\bar{x},\tilde{\eta})$ with $e_{\lambda }f(\bar{x})<e_{\lambda }f(x)<e_{\lambda }f(\bar{x})+\tilde{\nu}$, one has
\begin{align}
\tilde{\mu} d\big(0,\partial e_{\lambda }f(x)\big)\geq& \tilde{\mu} d\big(0,\lambda^{-1}(x-P_{\lambda}f(x))\big)
=\frac{\tilde{\mu}}{\lambda}\parallel x-y \parallel  \nonumber\\
\geq& d(x,\textrm{lev}_{\leq f(\bar{x})}f) \nonumber\\
\geq& d(x,\textrm{lev}_{\leq e_{\lambda}f(\bar{x})}e_{\lambda}f).   \label{4.4}
\end{align}
The first inequality follows from $\partial e_{\lambda }f(x)\subseteq \lambda^{-1}\big(x-P_{\lambda}f(x)\big) $, see \cite[ Example 10.32]{Rockafellar}; The last inequality holds with the fact $e_{\lambda}f(z)\leq f(z)\leq f(\bar{x})=e_{\lambda}f(\bar{x})$ for any $z\in \textrm{lev}_{\leq f(\bar{x})}f$ and then
$\textrm{lev}_{f(\bar{x})}f\subseteq \textrm{lev}_{\leq e_{\lambda}f(\bar{x})}e_{\lambda}f$.

If $\gamma \geq 1$, one has $d^{\gamma}(x,\textrm{lev}_{\leq f(\bar{x})}f)\leq \big(1+(\frac{\mu}{\lambda})^{\gamma^{-1}}\big)^{\gamma}\parallel x-y\parallel$
for any $x\in B(\bar{x},\tilde{\eta})$ with $e_{\lambda }f(\bar{x})<e_{\lambda }f(x)<e_{\lambda }f(\bar{x})+\tilde{\nu}$. Take $\bar{\mu}:=\lambda\big(1+(\frac{\mu}{\lambda})^{\gamma^{-1}}\big)^{\gamma}$, then for every such $x$, one has
\begin{align}
\bar{\mu} d\big(0,\partial e_{\lambda }f(x)\big)\geq& \bar{\mu} d\big(0,\lambda^{-1}(x-P_{\lambda}f(x))\big)
=\frac{\bar{\mu}}{\lambda}\parallel x-y \parallel  \nonumber\\
\geq& d^\gamma(x,\textrm{lev}_{\leq f(\bar{x})}f) \nonumber\\
\geq& d^\gamma(x,\textrm{lev}_{\leq e_{\lambda}f(\bar{x})}e_{\lambda}f).   \label{4.5}
\end{align}
Combining (\ref{4.4}) and (\ref{4.5}), the proof is completed.
\end{proof}

\begin{remark}\rm
If condition (ii) in Theorem \ref{th4.1} is replaced as follows: $f$ satisfies the LSEB at $\bar{x}$ with an exponent $\gamma= 0$ and a constant $\mu>0$. It is easy to see that $e_{\lambda }f$ satisfies the LSEB at $\bar{x}$ with the exponent 1 and the constant $\mu+\lambda$.

\end{remark}

Under suitable assumptions, we now consider the behaviour of LHEB via Moreau envelopes.

\begin{theorem} \label{th4.2}
Let $f:\mathbb{R}^n \rightarrow \bar{\mathbb{R}}$ be a proper and closed function, and $\bar{x}\in \textrm{dom}\,f$ with $0\in \partial e_{\lambda}f(\bar{x})$. Assume that the following conditions hold:
\begin{enumerate}
  \item [(i)] $f$ is prox-bounded with threshold $\lambda_f$ and $\lambda\in (0, \lambda_f )$;
  \item [(ii)] $f$ satisfies the LHEB at $\bar{x}$ with an exponent $\gamma> 0$ and a constant $\mu\geq 2\lambda$.
  \item [(iii)] For $\bar{\varepsilon}>0$ and any $x\in  B(\bar{x},\bar{\varepsilon})$ with $f(x)> f(\bar{x})$, there exists a point $y\in P_{\lambda }f(x)$ such that $f(y)\geq f(\bar{x})$.
\end{enumerate}
Then $e_{\lambda }f$ satisfies the LHEB at $\bar{x}$ with an exponent $\max\{2,\gamma\}$ and a constant $2^{\max\{1,\gamma-1\}}\mu$.
\end{theorem}

\begin{proof} By the condition (i) and Example 10.32 in \cite{Rockafellar}, for any $\bar{x}\in \textrm{dom}\,f$ with $0\in \partial e_{\lambda}f(\bar{x})$, we have $\bar{x}\in P_{\lambda}f(\bar{x})$ and $f(\bar{x})= e_{\lambda }f(\bar{x})$.

Since $f$ satisfies the LHEB at $\bar{x}$ with an exponent $\gamma> 0$ and a constant $\mu\geq 2\lambda$, there exists $0<\eta<1$ such that
\begin{align}
d^{\gamma}(x,\textrm{lev}_{\leq f(\bar{x})}f)\leq \mu \big(f(x)-f(\bar{x})\big)\label{4.7}
\end{align}
for any $ x\in B(\bar{x},\eta)$ with $ f(\bar{x})\leq f(x)$.

Taking $\varepsilon :=\min\{\frac{\eta}{2},\bar{\varepsilon}\}$. For any $x\in B(\bar{x},\varepsilon)$ with $e_{\lambda }f(\bar{x})<e_{\lambda }f(x)$, by $e_{\lambda }f(x)\leq f(x)$ and condition (iii), there exists a point $y\in P_{\lambda}f(x)$ such that $f(y)\geq f(\bar{x})$ and % $y\in P_{\lambda}f(x)$, one has
$$f(y)+\frac{1}{2\lambda}\parallel y-x \parallel^2\leq f(\bar{x})+\frac{1}{2\lambda}\parallel \bar{x}-x \parallel^2,$$
which means that $\parallel y-x \parallel\leq  \parallel x-\bar{x} \parallel<\frac{\eta}{2}$. Then we have $\parallel y-\bar{x} \parallel\leq \parallel y-x  \parallel+\parallel x-\bar{x} \parallel<\eta$, hence (\ref{4.7}) holds with $y$. Then for any $x\in B(\bar{x},\varepsilon)$ with $e_{\lambda }f(\bar{x})<e_{\lambda }f(x)$ and $\gamma\geq 2$, we have
\begin{align}
 d^{\gamma}(x,\textrm{lev}_{\leq e_{\lambda}f(\bar{x})}e_{\lambda}f) \leq& d^{\gamma}(x,\textrm{lev}_{\leq f(\bar{x})} f) \leq \big(d(y,\textrm{lev}_{\leq f(\bar{x})} f)+\|y-x\| \big)^{{\gamma}}\nonumber\\
\leq& 2^{\gamma-1}\big(d^{\gamma}(y,\textrm{lev}_{\leq f(\bar{x})} f)+\|y-x\|^{\gamma} \big)\nonumber\\
\leq& 2^{\gamma-1}\big(\mu\big(f(y)-f(\bar{x})\big)    +\|y-x\|^{\gamma} \big)\nonumber\\
=& 2^{\gamma-1}\mu\big(f(y)+\frac{1}{\mu}\|y-x\|^{\gamma}-f(\bar{x})\big)\nonumber\\
\leq& 2^{\gamma-1}\mu\big(f(y)+\frac{1}{2\lambda}\|y-x\|^{2}-f(\bar{x})\big)\nonumber\\
=& 2^{\gamma-1}\mu\big(e_{\lambda}f(x)-e_{\lambda}f(\bar{x})\big),\label{4.8}
\end{align}
where the first inequality holds with the fact $\textrm{lev}_{f(\bar{x})}f\subseteq \textrm{lev}_{\leq e_{\lambda}f(\bar{x})}e_{\lambda}f$, which follows from $e_{\lambda}f(x)\leq f(x)$ and $e_{\lambda}f(\bar{x})= f(\bar{x})$. The third inequality holds with the fact $(a+b)^\alpha\leq 2^{\alpha-1}(a^\alpha+b^\alpha)$ while $a,b\geq 0$ and $\alpha\geq 1$, since $g:x\mapsto x^\alpha$ is a convex function on $[0,\infty)$ with $\alpha\geq 1$. The fourth inequality holds with (\ref{4.7}). The last inequality holds with the condition $\mu\geq 2\lambda$, $\|y-x\|<1$ and $\gamma\geq 2$.

If $0<\gamma < 2$, for any $x\in B(\bar{x},\varepsilon)$ with $e_{\lambda }f(\bar{x})<e_{\lambda }f(x)$, one has
\begin{align}
 d^{2}(x,\textrm{lev}_{\leq e_{\lambda}f(\bar{x})}e_{\lambda}f) \leq& d^{2}(x,\textrm{lev}_{\leq f(\bar{x})} f) \leq \big(d(y,\textrm{lev}_{\leq f(\bar{x})} f)+\|y-x\| \big)^{{2}}\nonumber\\
\leq& 2\big(d^{2}(y,\textrm{lev}_{\leq f(\bar{x})} f)+\|y-x\|^{2} \big)\nonumber\\
\leq& 2\big(d^{\gamma}(y,\textrm{lev}_{\leq f(\bar{x})} f)+\|y-x\|^{2} \big)\nonumber\\
\leq& 2\big(\mu\big(f(y)-f(\bar{x})\big)    +\|y-x\|^{2} \big)\nonumber\\
%=& 2 \mu\big((f(y)+\frac{1}{\mu}\|y-x\|^{2}-f(\bar{x})\big)\nonumber\\
\leq& 2 \mu\big(f(y)+\frac{1}{2\lambda}\|y-x\|^{2}-f(\bar{x})\big)\nonumber\\
=& 2 \mu\big(e_{\lambda}f(x)-e_{\lambda}f(\bar{x})\big),\label{4.9}
\end{align}
the fourth inequality holds with $\gamma < 2$ and $\parallel y-\bar{x} \parallel<1$. Combining (\ref{4.8}) and (\ref{4.9}), the proof is completed.
\end{proof}

The following Proposition give the necessary conditions of (iii) in Theorem \ref{th4.2}.

\begin{proposition}\label{prop3.1}
Let $f:\mathbb{R}^n \rightarrow \bar{\mathbb{R}}$ be a proper and closed function and $\bar{x}\in \textrm{dom}\,f$ with $0\in \partial f(\bar{x})$. Assume that $f$ is prox-bounded with threshold $\lambda_f$ and $\lambda\in (0, \lambda_f )$. Suppose that either
\begin{enumerate}
  \item [(i)] $\bar{x}$ is a global optimal point; or
  \item [(ii)] $\bar{x}$ is a local optimal point and $f$ is prox-regular at $\bar{x}$ for $\bar{v}=0$.
  \end{enumerate}
Then, the condition (iii) in Theorem \ref{th4.2} holds.
\end{proposition}
\begin{proof} (i) holds naturally and we only consider (ii). By Proposition 13.37 in \cite{Rockafellar} we known that $P_{\lambda}f$ is satisfied a Lipschitz condition of rank $K$ on $\mathbb{R}^n$ and $P_{\lambda}f(\bar{x})=\bar{x}$. Then for $\bar{\varepsilon}>0$ and any $x\in  B(\bar{x},\bar{\varepsilon})$ with $f(x)> f(\bar{x})$,
\begin{align}
\|P_{\lambda}f(x)-P_{\lambda}f(\bar{x})  \|=\|y-x\| \leq K \|x-\bar{x}\|<K\bar{\varepsilon}, \nonumber
\end{align}
where $y=P_{\lambda}f(x)$. Since $\bar{x}$ is a local optimal point, one has that $f(y)\geq f(\bar{x})$. The proof is completed. \end{proof}

Consider the u-LSEB and u-HEB, we have the following results.
\begin{corollary} \label{coro 4.1}
Let $f:\mathbb{R}^n \rightarrow \bar{\mathbb{R}}$ be a proper and closed function and prox-bounded with threshold $\lambda_f$. Assume that the following conditions hold:
\begin{enumerate}
  \item [(i)] $f$ takes a constant value on a nonempty and compact set $\Lambda \subseteq \textrm{dom}\,f$;
  \item [(ii)] For any $\bar{x}\in \Lambda$, $0\in \partial e_{\lambda}f(\bar{x})$ with $\lambda\in (0, \lambda_f )$;
  \item [(iii)] $f$ satisfies the u-LSEB on $\Lambda$ with an exponent $\gamma> 0$ and a constant $\mu>0$.
\end{enumerate}
Then $e_{\lambda }f$ satisfies the u-LSEB on $\Lambda$ with an exponent $\max\{1,\gamma\}$ and a constant $\lambda\big(1+(\frac{\mu}{\lambda})^{\gamma^{-1}}\big)^{\max\{1,\gamma\}}$.
\end{corollary}
\begin{proof} By condition (iii) and Definition \ref{def2.1}, $f$ satisfies the LSEB at every point $\bar{x}\in \Lambda$ with an exponent $\gamma $ and a constant $\mu$. Then by condition (ii) and the Theorem \ref{th4.1}, it is easy to see that $e_{\lambda }f$ satisfies the LSEB at every point $\bar{x}\in \Lambda$ with an exponent $\max\{1,\gamma\}$ and a constant $\lambda\big(1+(\frac{\mu}{\lambda})^{\gamma^{-1}}\big)^{\max\{1,\gamma\}}$, which means that $e_{\lambda }f$ satisfies the u-LSEB on $\Lambda$ with an exponent $\max\{1,\gamma\}$ and a constant $\lambda\big(1+(\frac{\mu}{\lambda})^{\gamma^{-1}}\big)^{\max\{1,\gamma\}}$. The proof is completed. \end{proof}

We omit the proof of Corollary \ref{coro 4.2}, which proof is similar to Corollary \ref{coro 4.1}.

\begin{corollary} \label{coro 4.2}
Let $f:\mathbb{R}^n \rightarrow \bar{\mathbb{R}}$ be a proper and closed function and prox-bounded with threshold $\lambda_f$. Assume that the following conditions hold:
\begin{enumerate}
  \item [(i)] $f$ takes a constant value on a nonempty and compact set $\Lambda \subseteq \textrm{dom}\,f$;
  \item [(ii)] For any $\bar{x}\in \Lambda$, $0\in \partial e_{\lambda}f(\bar{x})$ with $\lambda\in (0, \lambda_f )$;
  \item [(iii)] $f$ satisfies the u-HEB on $\Lambda$ with an exponent $\gamma> 0$ and a constant $\mu\geq 2\lambda$;
  \item [(iv)] For $\bar{\varepsilon}>0$ and any $x\in   B(\Lambda,\bar{\varepsilon}):=\{y:d(y,\Lambda)< \bar{\varepsilon}\}$ with $f(x)> f(\bar{x})$, there exists a point $y\in P_{\lambda }f(x)$ such that $f(y)\geq f(\bar{x})$.
\end{enumerate}
Then $e_{\lambda }f$ satisfies the u-HEB on $\Lambda$ with an exponent $\max\{2,\gamma\}$ and a constant $2^{\max\{1,\gamma-1\}}\mu$.
\end{corollary}

From Remark 5.1 (i) in \cite{Li 6}, if $f$ is a KL function with exponent $\alpha\in (\frac{1}{2},1]$ and $\inf f > -\infty$, then $e_{\lambda }f$ is a KL function with exponent $\alpha$.
The following example shows that our conclusion is more applicable than Remark 5.1 (i).
\begin{example}
Let $f: \mathbb{R}\rightarrow \mathbb{R}$ be given by
\begin{align*}
\begin{split}
f(x ):= \left \{
\begin{array}{ll}
     0                                             & \textrm{if}\,\, x\leq 0,\\
     x^2+\frac{1}{n}-\frac{1}{n^2}                 & \textrm{if}\,\, \frac{1}{n}<x\leq \frac{1}{n-1},n=3,4,...,\\
     x^2+\frac{1}{4}                               & \textrm{if}\,\, x> \frac{1}{2}.\\
\end{array}
\right.
\end{split}
\end{align*}
It follows from the Example 3.19 in \cite{Kruger 2019 2-23} that $f$ is proper and closed function, and from Example 2.2 in \cite{Bai 1} that $f$ satisfies the LSEB at $\bar{x}=0$ with the exponent $1$ and the constant $\frac{1}{2}$. However, $f$ does not have KL property at $0$ with any exponent $\alpha\in  [0, 1)$, since, for $x_n=\frac{1}{n-1}$ with $n$ sufficiently large, one has
$$\big(f(x_n )-f(0)\big)^{-\alpha}d\big(0,\partial f(x_n )\big)\leq (\frac{1}{n})^{-\alpha}\frac{2}{n-1}=\frac{2n^{\alpha}}{n-1}\rightarrow 0 \,\,\textrm{as}\,\, n\rightarrow \infty.$$

Now, we consider the following function:
\begin{align*}
\begin{split}
e_{ \frac{1}{2}}f(x):= \left \{
\begin{array}{ll}
     0                                           & \textrm{if}\,\, x\leq 0,\\
    \frac{1}{2}x^2+\frac{1}{n}-\frac{1}{n^2}                 & \textrm{if}\,\, \frac{1}{n}<x\leq \frac{1}{n-1},n=3,4,...,\\
     \frac{1}{2}x^2+\frac{1}{4}                               & \textrm{if}\,\, x> \frac{1}{2}.\\
\end{array}
\right.
\end{split}
\end{align*}
It is easy to see that $e_{ \frac{1}{2}}f$ is a proper and closed function and $f$ satisfies the conditions (i) and (ii) in Theorem \ref{th4.1}, which means that $e_{ \frac{1}{2}}f$ satisfies the LSEB at $\bar{x}=0$ with the exponent $1$ and the constant $1$. However,$e_{ \frac{1}{2}}f$ does not have KL property at $0$ with any exponent $\alpha\in  [0, 1)$, since, for $x_n=\frac{1}{n-1}$ with $n$ sufficiently large, one has
\begin{align}
\big(e_{ \frac{1}{2}}f(x_n )-e_{ \frac{1}{2}}f(0)\big)^{-\alpha}d\big(0,\partial e_{ \frac{1}{2}}f(x_n )\big)&=\big(\frac{1}{2(n-1)^2}+\frac{1}{n}-\frac{1}{n^2}\big)^{-\alpha}\frac{1}{n-1}\nonumber\\
&\leq (\frac{1}{n}-\frac{1}{2n^2})^{-\alpha}\frac{1}{n-1}\nonumber\\
&\leq (\frac{5}{6n})^{-\alpha}\frac{1}{n-1}=\frac{(1.2n)^{\alpha}}{n-1}\rightarrow 0 \,\,\textrm{as}\,\, n\rightarrow \infty.\nonumber
\end{align}

\end{example}

The following example shows that Theorem \ref{th4.2} is more applicable than Theorem \ref{th4.1}.

\begin{example}
Let $f: \mathbb{R} \rightarrow \mathbb{R}$ be given by
\begin{align*}
\begin{split}
f(x):= \left \{
\begin{array}{ll}
   x^{2}(2+\cos{\frac{1}{x}} )                  & \textrm{if}\,\, x\neq 0,\\
      0                                               & \textrm{if}\,\, x= 0.\\
\end{array}
\right.
\end{split}
\end{align*}
It easy to see that closed function $f$ satisfies the LHEB at $\bar{x}=0$ with the exponent $2$ and the constant $1$. But $f$ not satisfies the LSEB at $0$ with any exponent $\gamma\in  [0, +\infty)$, since, for any $x$ sufficiently near $0$ with $x\neq 0$, we have
$$\nabla f(x)=4x+2x\cos\frac{1}{x}+\sin\frac{1}{x}.$$
Picking $x_k^1=\frac{1}{2k\pi}$ and $x_k^2=\frac{1}{2k\pi+1.5\pi}$, for sufficiently large $k$, one has $\nabla f(x_k^1)=\frac{3}{k\pi}>0$ and $\nabla f(x_k^2)=\frac{4}{2k\pi+1.5\pi}-1<0$, which means that there is a sequence $\{x_k\}$ converging to $0$ with $\nabla  f(x_k)=0$ for all $k\in N$.

Now, we consider the following function:
\begin{align*}
\begin{split}
e_{ \frac{1}{2}}f(x):= \left \{
\begin{array}{ll}
   \inf_{y\in \mathbb{R}^n}\big\{y^{2}(2+\cos{\frac{1}{y}} ) +\|y-x\|^2  \}                & \textrm{if}\,\, x\neq 0,\\
      0                                                & \textrm{if}\,\, x= 0.\\
\end{array}
\right.
\end{split}
\end{align*}
Taking into account that $0$ is a global optimal point and $\mu= 2\lambda=1$. From proposition \ref{prop3.1} and Theorem \ref{th4.2}, we have that $e_{ \frac{1}{2}}f$ satisfies the LHEB at $0$ with the exponent 2 and the constant 2. Noting that for any $x$ near $0$ with $x\neq 0$, $P_{ \frac{1}{2}}f(x)\subseteq [\frac{x}{4},\frac{x}{2}]$. So we can assume that $ax \in P_{ \frac{1}{2}}f(x)$ with $a\in [ \frac{1}{4}, \frac{1}{2}]$. For any $x$ sufficiently near $0$, if $x\neq 0$, we have
$$\nabla e_{ \frac{1}{2}}f(x)=(6a^2-4a+2)x+2a^2x\cos\frac{1}{ax}+a\sin\frac{1}{ax}.$$
Picking $x_k^1=\frac{1}{2ak\pi}$ and $x_k^2=\frac{1}{2ak\pi+1.5a\pi}$, for sufficiently large $k$, one has $\nabla f(x_k^1)=\frac{8a^2-4a+2}{2ak\pi}>0$ and $\nabla f(x_k^2)=\frac{6a^2-4a+2}{2ak\pi+1.5a\pi}-a<0$, which means that there is a sequence $\{x_k\}$ converging to $0$ with $\nabla  f(x_k)=0$ for all $k\in N$. Hence $e_{ \frac{1}{2}}f$ not satisfies the LSEB at $0$ with any exponent $\gamma\in  [0, +\infty)$.
\end{example}

\section{Concluding remarks }\label{5}

In this paper, we first establish the equivalence of three types of error bounds: u-KL property, u-LSEB and u-HEB for prox-regular functions. Then we show that the behavior of the LSEB (resp. LHEB) exponent and constant is via Moreau envelopes under suitable assumptions. Intuitively, our proof procedure should be similar to that of the literature \cite{Li 6}, but it is difficult to get a conclusion similar to Lemma 5.1 in \cite{Li 6}. Thus, in Theorems \ref{th4.1} and \ref{th4.2}, we prove the LSEB and LHEB property of $e_\lambda f$ directly, without having to go through the inf-projection function $\inf_{y\in\mathbb{R}^m}F(x,y)$. However it is worth noting that condition (iii) cannot be omitted in Theorem \ref{th4.2}.

One future research direction is to develop the behavior of the LSEB (resp. LHEB) exponent and constant via the inf-projection function. Another direction is to find some functions that satisfy LSEB but not KL inequality and discuss the convergence of the algorithm.

\vskip 6mm
\noindent{\bf Acknowledgments}

\noindent  % The author is grateful to the reviewers for useful suggestions which improved the contents of this paper.

This research was supported by the National Natural Science Foundation of China (Grant No. 11971078)


\begin{thebibliography}{99}%in appearance order.

\bibitem{Mordukhovich 2006 2-27} B.S. Mordukhovich, Variational Analysis and Generalized Differentiation I: Basic Theory, vol. 330, Springer, 2006.
\bibitem{Rockafellar} R.T. Rockafellar, R.J.-B. Wets, Variational Analysis. Springer, 1998.
\bibitem{Attouch 2009 6-15-1} H. Attouch, J. Bolte, On the convergence of the proximal algorithm for nonsmooth functions involving analytic features. Math. Program. 116 (2009), 5-16.
\bibitem{Attouch 2013} H. Attouch, J. Bolte,  B.F. Svaiter, Convergence of descent methods for semi-algebraic and tame problems: proximal algorithms, forward-backward splitting, and regularized Gauss-Seidel methods. Math. Program. 137 (2013), 91-129.
\bibitem{Burke 2002 1-10} J.V. Burke, S.E. Deng, Weak sharp minima revisited. Part I: basic theory. Control Cybern. 31 (2002), 439-469.
\bibitem{Bolte 2010} J. Bolte, A. Daniilidis, O. Ley, L. Mazet, Characterizations of {\L}ojasiewicz inequalities: subgradient flows, talweg, convexity. Trans. Am. Math. Soc. 362 (2010), 3319-3363.
\bibitem{Bolte 2014 6-15} J. Bolte, S. Sabach, M. Teboulle, Proximal alternating linearized minimization for nonconvex and nonsmooth problems. Math. Program. 146 (2014), 459-494.
\bibitem{Drusvyatskiy 3} D. Drusvyatskiy, A.S. Lewis, Error Bounds, quadratic growth, and linear convergence of proximal methods. Math. Oper. Res. 43 (2018), 919-948.
\bibitem{Li 5}  G.Y. Li, T.K. Pong, Calculus of the exponent of Kurdyka-Lojasiewicz inequality and its applications to linear convergence of first-order methods. Found. Comput. Math. 18 (2018), 1199-1232.
\bibitem{Li MH 2020} M.H. Li, C.R. Chen, S.J. Li, Error bounds of regularized gap functions for nonmonotone Ky Fan inequalities. J. Ind. Manag. Optim. 16 (2020), 1261-1272.
\bibitem{Pang 1997 2-31} J.S. Pang, Error bounds in mathematical programming. Math. Program. 79 (1997), 299-332.
\bibitem{Li 6} Y. Peiran, G.Y. Li, T.K. Pong, Kurdyka-{\L}ojasiewicz exponent via inf-projection. Found. Comput. Math. 22 (2022), 1171-1217.
\bibitem{Bolte 2017 1-7} J. Bolte, T.P. Nguyen, J. Peypouquet, W.S. Bruce, From error bounds to the complexity of first-order descent methods for convex functions. Math. Program. 165 (2017), 471-507.
\bibitem{Bai 1} S.X. Bai, M.H. Li, C.W. Lu, D.L. Zhu, S. Deng, The equivalence of three types of error bounds for weakly and approximately convex functions. J. Optimiz. Theory App. 194 (2022), 220-245.
\bibitem{Li 2016} G.Y. Li, T.K. Pong, Douglas-Rachford splitting for nonconvex optimization with application to nonconvex feasibility problems. Math. Program. 159 (2016), 371-401.
\bibitem{Lojasiewicz 1963} S. {\L}ojasiewicz, Une propri\'{e}t\'{e} topologi que des sous-ensembles analytiques r\'{e}els, in Les \'{E}quationsaux D\'{e}riv\'{e}es Partielles, \'{E}ditions du Centre National de la Recherche Scientifique, Paris, 117 (1963), 87-89.
\bibitem{zhu 2} D.L. Zhu, S. Deng, M.H. Li, L. Zhao, Level-set subdifferential error bounds and linear convergence of Bregman proximal gradient method. J. Optimiz. Theory App. 189 (2021), 889-918.
\bibitem{Kruger 2019 2-23} A.Y. Kruger, M.A. L\'{o}pez, X.Q. Yang, J.X. Zhu, H\"{o}lder error bounds and H\"{o}lder calmness with applications to convex semi-infinite optimization. Set-valued Var. Anal. 27 (2019), 995-1023.
\bibitem{Burke 1993} J.V. Burke, M.C. Ferris, Weak sharp minima in mathematical programming. SIAM J. Control Optim. 31 (1993), 1340-1359.
\bibitem{Ioffe 2017} A.D. Ioffe, Variational Analysis of Regular Mappings. Theory and Applications. Springer Monographs in Mathematics. Springer, 2017.
\bibitem{Ngai 2017} H.V. Ngai, N.H. Tron, P.N. Tinh, Directional H\"{o}lder metric subregularity and application to tangent cones. J. Convex Anal. 24 (2017), 417-457.
\bibitem{Yao 2016} J.C. Yao, X.Y. Zheng, Error bound and well-posedness with respect to an admissible function. Appl. Anal. 95 (2016), 1070-1087.
\bibitem{Moreau 1962} J.J. Moreau, Fonctions convexes duales et points proximaux dans un espace hilbertien, C. R. Acad. Sci. Paris. 255 (1962), 2897-2899.
\bibitem{Rockafellar 2009} A.L. Dontchev, R.T. Rockafellar, Implicit Functions and Solution Mappings. Springer, Berlin, 2009.
\bibitem{Jourani 2014} A. Jourani, L. Thibault, D. Zagrodny, Differential properties of the Moreau envelope. J. Funct. Anal. 266 (2014). 1185-1237.
\bibitem{Kan 2012} C. Kan, W. Song, The Moreau envelope function and proximal mapping in the sense of the Bregman distance. Nonlinear Anal.-Theor. 75 (2012), 1385-1399.
\bibitem{Perez-Aros 2021} P. P\'{e}rez-Aros, E. Vilches, Moreau envelope of supremum functions with applications to infinite and stochastic programming. SIAM J. Optimiz. 31 (2021), 1635-1657.
\bibitem{Rockafellar 1976} R.T. Rockafellar, Monotone operators and the proximal point algorithm. SIAM J. Control Optim. 14 (1976), 877-898.









































\end{thebibliography}
\end{document}